\documentclass{amsart}
\usepackage{graphicx} 
\usepackage{amssymb}
\usepackage{amsfonts}
\usepackage{amsmath}
\usepackage{amsrefs}
\usepackage{amsthm}
\usepackage{array}
\usepackage{graphicx}
\usepackage{graphics}
\usepackage{framed}
\usepackage{amsxtra}
\usepackage{tabularx}
\usepackage{lipsum}
\usepackage{mathtools}
\usepackage{enumerate}
\usepackage{hyperref}
\usepackage{setspace}
\usepackage{multicol}
\usepackage{comment}
\usepackage{xcolor}
\usepackage{thmtools}
\usepackage{nameref}
\usepackage{listings}
\usepackage{longtable,booktabs,array}
\usepackage[top=1in, bottom=1in, left=1in, right=1in]{geometry}
\usepackage{tikz}
\usepackage{cleveref}
\usetikzlibrary{matrix,decorations.pathreplacing,calc,positioning,backgrounds}

\theoremstyle{plain} 
\newtheorem{theorem}{Theorem}[section]
\newtheorem{lemma}[theorem]{Lemma}

\newtheorem{proposition}[theorem]{Proposition}

\newtheorem{construct}[theorem]{Construction}
\theoremstyle{definition}
\newtheorem{definition}[theorem]{Definition}

\newtheorem{exam}[theorem]{Example}

\newtheorem{remark}[theorem]{Remark}

\DeclareMathOperator{\reg}{reg}
\DeclareMathOperator{\supp}{supp}
\DeclareMathOperator{\lcm}{lcm}
\DeclareMathOperator{\opb}{b}
\DeclareMathOperator{\pd}{pd}
\usepackage{graphicx} 

\title{Minimal Resolutions of $k$-divisible and $(k-1)$-dominant Ideals}
\author{Patrick Welte}
\address{Department of Mathematics and Statistics, Binghamton University}
\email{pwelte@binghamton.edu}

\author{Guillermo Alesandroni}
\address{Mathematics and Statistics Department, California State University, Chico}
\email{galesandroni@csuchico.edu}

\begin{document}
\maketitle

\begin{abstract}
    We introduce invariants we call $k$-divisibility and $d$-dominance, which are combinatorial constraints on the minimal generators of monomial ideals. We then prove that $k$ (respectively, $d$) is an upper bound (respectively, lower bound) for the projective dimension of a $k$-divisible (respectively, $d$-dominant) monomial ideal. This leads us to an explicit minimal resolution of $k$-divisible, $d$-dominant ideals when $d=k-1$ via a Lyubeznik resolution. This framework allows us to prove a conjecture made by Alesandroni. Using these invariants, we also provide a characterization of monomial ideals whose minimal free resolution is attainable from the Taylor resolution by performing exactly one consecutive cancellation. We also prove that for all $q,k,d$ in the range $1\leq d<k \leq q-1$, there exist monomial ideals that are $k$-divisible and $d$-dominant with $q$ minimal generators.
\end{abstract}

\section{Introduction to $k$-divisibility and $d$-dominance}

We begin with the definitions of $k$-divisibility and $d$-dominance, with an example illustrating some finer details that will be important later. Section $2$ provides some background on a few concepts and facts from commutative and homological algebra we required; these include most importantly the definitions of Taylor and Lyubeznik resolutions. Section $3$ is where we prove that if an ideal satisfies the combinatorial property of being $k$-divisible (respectively, $d$-dominant), then $k$ is an upper bound (respectively, $d$ is a lower bound) on the projective dimension of the ideal. Here we also derive the minimal resolutions of the $k$-divisible, $(k-1)$-dominant ideals, resolve a conjecture of Alesandroni, and characterize the so-called $1$-cancellation ideals in terms of our invariants. In section $4$ we offer a construction of $k$-divisible, $d$-dominant ideals for each suitable $k,d$. Throughout our writing, $M$ will be a monomial ideal of $S=\mathbb{F}[x_1, \cdots, x_n]$, where $\mathbb{F}$ is an arbitrary field, and projective dimension and regularity are applied to $S/M$. Finally, when we write $M=(m_1,...,m_q)$ we mean that the minimal generating set of the ideal $M$ is exactly $\{m_1,...,m_q\}$. 

\begin{definition}
    For $1\leq k \leq q$, we say that $M=(m_1,...,m_q)$ is $\mathbf{k}$\textbf{-divisible} if it has a minimal generator that divides the least common multiple of every $k$ minimal generators. That is, $M$ is $k$-divisible if there exists $m \in \{m_1,...,m_q\}$ such that $m|\lcm\{m_{i_1},...,m_{i_k}\}$ for every $k$-set $\{m_{i_1},...,m_{i_k}\} \subset \{m_1,...,m_q\}$. Such an $m$ will be called a \textbf{witness} of the $k$-divisibility of $M$.
\end{definition}

\begin{definition}
   For $1\leq d \leq q-1$, we say that $M=(m_1,...,m_q)$ is $\mathbf{d}$\textbf{-dominant} if no minimal generator divides the lcm of $d$ others. That is, $M$ is $d$-dominant if for every $(d+1)$-set $\{m_{j_1},...,m_{j_{d+1}}\} \subset \{m_1,...,m_q\}$ and every $m\in \{m_{j_1},...,m_{j_{d+1}}\}$, $m\nmid \lcm( \{m_{j_1},...,m_{j_{d+1}}\} \setminus\{m\})$.
\end{definition}

Observe that if an ideal $M$ is $k$-divisible and $d$-dominant, then for each $k'$ and $d'$ satisfying $k< k' \leq q$ and $1\leq d'<d$, $M$ is also $k'$-divisible and $d'$-dominant. This fact motivates the definitions of the following invariants: $\kappa(M)= \min\{k \in \mathbb{N}: M \text{ is } k \text{-divisible}\}$ and $\delta(M)=\max\{d \in \mathbb{N}: M \text{ is } d \text{-dominant}\}$ \textemdash{} we use this notation later in section $4$ when we construct ideals with these properties. Notice that $M$ is always $q$-divisible with any generator serving as a witness, so $\kappa(M)$ is well-defined. Likewise, $\delta(M)$ is well-defined since $M$ must be at least $1$-dominant by the definition of a minimal generating set of a monomial ideal. We make one more remark on how $\kappa(M)$ and $\delta(M)$ relate to each other. If $M$ is $k$-divisible and $(k-1)$-dominant, then $\kappa(M)=k$ and $\delta(M)=k-1$. If $M$ were $(k-1)$-divisible, then there must exist $m \in \{m_1,...,m_q\}$ such that $m$ divides the lcm of every $(k-1)$-collection of generators. But this clearly contradicts $(k-1)$-dominance, which requires that no generator can divide the lcm of any $(k-1)$-collection of other generators. This means that $k$ is the least natural number for which $M$ is $k$-divisible, that is $\kappa(M)=k$ (it is trivial to check also that $\delta(M)=k-1$). Finally, notice that $d$-dominance is not simply the negation of $d$-divisibility. A monomial ideal fails to be $d$-divisible when it has no minimal generator that divides the lcm of every $d$-collection but $d$-dominance is stronger, requiring that no minimal generator divides the lcm of any $d$-collection of other generators. These notions may come apart. Consider the ideal, $M=(xy,xu,yv,z)$. No generator divides the lcm of every pair of generators, hence $M$ is not $2$-divisible. But $xy|\lcm(xu,yv)$, and so $M$ is also not $2$-dominant. We work through a similar example in more detail to further illustrate these ideas.

\begin{exam}
    Define $M=(xyz,xu,yv,zw)$. Since $xyz|\lcm(xu,yv,zw)$, we say that $M$ is $3$-divisible. Observe that $xyz$ does not divide the lcm of any pair of generators from the other three $xu,yv,zw$ and, similarly, none of $xu,yv,zw$ divide the lcm of any two other generators from $xyz,xu,yv,zw$ since each possesses a unique variable $p_i$. Hence, $M$ is $2$-dominant and since $M$ is $3$-divisible, it follows that $\kappa(M)=3$ and $\delta(M)=2$.
\end{exam}

\section{Commutative algebra background}

The definitions here are concise forms of Mermin's definitions in \cite{Merm}. We have left out some unnecessary details, but interested readers should study the excellent introduction to these resolutions in full in Mermin's paper. Throughout, we identify subsets of $G$ with faces of $\Delta_M$, using these words synonymously.

\begin{definition}
    For a subset $\sigma \subset G$, set $\lcm(\sigma)= \lcm\{m_i: m_i \in \sigma\}$. For each such $\sigma$, define a formal symbol $[\sigma]$, called a Taylor symbol, with multidegree equal to $\lcm(\sigma)$. For each $i$, set $T_i$ equal to the free $S$-module with basis $\{[\sigma]:|\sigma|=i\}$ given by the symbols corresponding to subsets of size $i$. Define $\phi_{-1}: T_{0}\xrightarrow{}S/M$ by $\phi_{-1}(f[\emptyset])=f$. Otherwise,  we construct $\phi_i: T_{i+1}\xrightarrow{}T_i$ as follows. Given $\sigma=\{m_{j_1},...,m_{j_i}\}$ and $\tau = \sigma \setminus \{m_{j_t}\}$, we set the sign $\epsilon_\tau^\sigma$ equal to $1$ if $t$ is odd and $-1$ if $t$ is even. We call such an $\tau$ a facet of $\sigma$. Finally, we set : $$\phi_\sigma = \sum_{\tau \text{  a facet of } \sigma}\epsilon_\tau^\sigma \frac{\lcm(\sigma)}{\lcm(\tau)}[\tau]$$ and define $\phi_i:T_{i+1} \xrightarrow{}T_i$ by extending the various $\phi_\sigma$. The Taylor resolution of $M$ is the complex $$\mathbb{T}_M: 0 \xrightarrow{}T_q \xrightarrow{\phi_{q-1}}\cdots \xrightarrow{\phi_1}T_1 \xrightarrow{\phi_0}T_0 \xrightarrow{}S/M \xrightarrow{}0$$
\end{definition}

\begin{definition}
    Fix any total ordering $<$ on the minimal generators of $M$: $m_{i_1}<...<m_{i_q}$. For a monomial $u$, define $\min(u)=\min_<(m_i: m_i \text{ divides } u)$ and for each $\sigma \in \Delta_M$ let $\min(\sigma)=\min(\lcm(\sigma))$ -- notice that $\min(\sigma)$ is a monomial. We say that $\sigma$ is rooted if every non-empty subset $\tau \subset \sigma$ satisfies $\min(\tau) \in \tau$ (in particular, $\min(\sigma) \in \sigma$). By construction the set $\Lambda_{M,<}=\{ \sigma \in \Delta_M : \sigma \text{ is rooted }\}$ is a simplicial complex; we call it the Lyubeznik simplicial complex associated to $M$ and $<$. The associated algebraic chain complex $\mathbb{L}_{M,<}$ is called a Lyubeznik resolution of $M$. Suppose $M=(m,m_1,...,m_{q-1})$ is $k$-divisible with witness $m$. Later we will examine a ``witness-first" Lyubeznik resolution, which is simply a resolution associated with any total ordering of the minimal generators that has $m$ as the least element. Given a witness-first ordering $<$, we denote the associated Lyubeznik resolution by $\mathbb{L}_{M,<}$. We will shorten this by writing just $\mathbb{{L}_M}$ to mean the witness-first resolution (when each witness-first resolution is isomorphic). That is, if, once a witness generator is fixed as the least element, the associated resolution is independent of the ordering of the remaining generators, then we use $\mathbb{{L}_M}$ to indicate this resolution. This rather technical point will become clearer after \cref{thm:big-lyub}.
\end{definition}

In general the Taylor resolution of a monomial ideal is not the minimal resolution; indeed, it is minimal if and only if the ideal is ``dominant" as shown in \cite{Ales3} (we will return to the notion of dominant ideals later). However, starting from the Taylor resolution we can always obtain the minimal resolution of a monomial ideal by performing so-called consecutive cancellations to remove redundant basis elements of the resolution. We make an informal definition below. Consecutive cancellation is a standard procedure and greater treatment is beyond our scope. For more details see \cite{Peeva}.

\begin{definition}
    Given $i \in \{1,2,...,r\}$, if $[m_1,...,m_r]$ and $[m_1,...,\widehat{m_i},...,m_r]$ are basis elements of a free resolution of $M$, then $[m_1,...,\widehat{m_i},...,m_r]$ is called a facet of the face $[m_1,...,m_r]$. By consecutive cancellation, we mean the process simplifying a free resolution of $M$ by removing a face and facet of equal multidegree in adjacent homological degree. It is known by \cite{Ales3} and \cite{Peeva} that the minimal resolution of a monomial ideal $M$ can be obtained from $\mathbb{T}_M$ by performing sequentially all possible consecutive cancellations. 
\end{definition}

\begin{definition}
    A Scarf face is a subset $A$ of the minimal generating set of $M$ such that $\lcm(A)$ is not achieved by any other face. That is, $A$ is Scarf if there is no face $B \neq A$ with $\lcm(B)=\lcm(A)$.
\end{definition}

\begin{lemma}\label{lem:scarf}
    If $M$ is $d$-dominant, then each set $\sigma\subset \{m_1,...,m_q\}$ with $1\leq|\sigma|\leq d$ is Scarf, and $\opb_j(S/M)=\binom{q}{j}$ for each $0\leq j \leq d$.
\end{lemma}
\begin{proof}
    Let $\sigma\subset \{m_1,...,m_q\}$ be arbitrary and $1\leq|\sigma|\leq d$. Recall that $d$-dominance implies $t$ dominance for each $t \in \{1,2,...,d-1\}$. Suppose for a contradiction that another face $\tau\subset \{m_1,...,m_q\}$ with $\tau \neq \sigma$ satisfies $\lcm(\sigma)=\lcm(\tau)$. If $(\tau \setminus \sigma) \neq \emptyset$, then choose $a \in (\tau \setminus \sigma)$. Then $a|\lcm(\tau)=\lcm(\sigma)$. This is impossible by $d$-dominance. Now if $(\tau \setminus \sigma) = \emptyset$, then since $\tau \neq \sigma$ it must be that $|\tau|<d$. Then we may choose $a' \in (\sigma \setminus \tau)$, and then $a'|\lcm(\sigma)=\lcm(\tau)$, which is again impossible by $d$-dominance. 
    
    Either way results in a contradiction, hence $\sigma$ must be Scarf. Since each face in this range is Scarf, no consecutive cancellation in the Taylor resolution below homological degree $d$ is possible, hence $\opb_j(S/M)=\binom{q}{j}$ for each $0\leq j \leq d$.
\end{proof}

We will also use the terminology of Betti numbers, which we introduce below. Unless stated explicitly, we will always be referring to total Betti numbers.

\begin{definition}
    The total Betti numbers of $S/M$ are the ranks of the modules in the minimal free resolution of $S/M$. If $\pd(S/M)=n$, then there are exactly $n+1$ Betti numbers and the $i$'th is the rank of the corresponding $i$'th module in the minimal resolution. We indicate the $i$'th Betti number of $S/M$ by writing $\opb_i(S/M)$.
\end{definition}

\section{Minimal Resolutions of $k$-Divisible, $d$-Dominant Ideals}

The dominance and divisibility properties we have defined are related to minimal free resolutions. In particular, if $M$ is $d$-dominant and $k$-divisible, then the dominance property gives a lower bound on the projective dimension of $M$ and fixes the first $d$ Betti numbers and the divisibility property gives an upper bound on the projective dimension and $k$'th Betti number. We begin by showing that $k$ upper bounds the projective dimension of $k$-divisible monomial ideals by choosing an ordering of the generators that makes the witness of $k$-divisibility the least element. We agree to make the following convention for the rest of the paper: every free resolution $\mathbb{F}$ of $S/M$ is assumed to be obtained from $\mathbb{T}_M$ by consecutive cancellations; accordingly, basis elements are Taylor symbols, with homological degree $|A|$ and multidegree $\lcm(A)$.

\begin{lemma}\label{lem:root}
    Let $M=(m,m_1,...,m_{q-1})$ be $k$-divisible with witness $m$, and let $<$ be any witness-first ordering of the generators. If $\sigma$ is a rooted face, then $|\sigma|\leq k$. Moreover, if $|\sigma|=k$ then $m \in \sigma$.
\end{lemma}
\begin{proof}
    Suppose $\sigma$ is a face of size $>k$. If $m\notin \sigma$, then $k$-divisibility implies that $m|\lcm(\sigma)$, which means $\sigma$ must not be rooted. If $m\in \sigma$, define $\sigma'=\sigma \setminus\{m\}$ which means $\sigma=\sigma'\cup \{m\}$. Since we assumed $|\sigma|>k$, we know $|\sigma'|\geq k$. Now since $m$ is a witness of $k$-divisibility, $m|\lcm(\sigma')$. But since $m \notin \sigma'$, we know $\sigma'$ is not rooted, hence neither is $\sigma$. 

    Last, if $|\sigma|=k$ and $m \notin \sigma$, then $m|\lcm(\sigma)$ by the definition of $k$-divisibility. So here also, $\sigma$ is not rooted.
\end{proof}

\begin{theorem}\label{thm:kdiv-bound}
    If $M$ is $k$-divisible, then $\pd(S/M) \leq k$ and $\opb_k(S/M) \leq \binom{q-1}{k-1}$.
\end{theorem}

\begin{proof}
    Consider a $k$-divisible ideal: $M=(m,m_1,...,m_{q-1})$, with $m$ a witness of $k$-divisibility. Then consider the ordering $m<m_1<...<m_{q-1}$. In the associated Lyubeznik resolution, every size $k$ face must contain $m$ by \cref{lem:root}. Hence it must be that $\opb_k(S/M) \leq \binom{q-1}{k-1}$. Last, \cref{lem:root} also tells us that $\pd(S/M)\leq k$ since there are no rooted faces of size strictly larger than $k$.
\end{proof}

The notion of $d$-dominance similarly produces a lower bound on the projective dimension of $M$; it will be our subsequent goal to pair these properties to consider a case where they fully determine the minimal resolution. We first provide a proof of a standard fact about the Betti numbers of monomial ideals.

\begin{lemma}\label{lem:alt-sum}
    For any monomial ideal $M\neq 0$, $\sum_{i\geq 0}(-1)^i\opb_i(S/M)=0$.
\end{lemma}

\begin{proof}
    Recall that for any $q>0$, $\sum_{i=0}^q(-1)^i\binom{q}{i}=0$. Notice that if $M=(m_1,...,m_q)$, then $\sum_{i=0}^q(-1)^i\binom{q}{i}$ is precisely the alternating sum of the ranks of the modules in the Taylor resolution of $S/M$. The minimal free resolution of $S/M$ may be obtained from its Taylor resolution by means of consecutive cancellation, and at each stage of obtaining the minimal resolution from the Taylor resolution, consecutive cancellation removes exactly two basis elements from the resolution. After one cancellation in the Taylor resolution, say between elements in homological degrees $k$ and $k+1$, the sum is preserved: $$\sum_{i=0}^{k-1}(-1)^i\binom{q}{i}+(-1)^k\Big{(}\binom{q}{k}-1\Big{)}+(-1)^{k+1}\Big{(}\binom{q}{k+1}-1\Big{)}+\sum_{i=k+2}^{q}(-1)^i\binom{q}{i}=\sum_{i=0}^{q}(-1)^i\binom{q}{i}=0$$
    At each stage of obtaining the minimal resolution from the Taylor resolution via consecutive cancellation, the alternating sum of the ranks of the modules in the resolution is $0$. Hence, $\sum_{i=0}^d(-1)^i\opb_i(S/M)=0$.
\end{proof} 

\begin{proposition}\label{prop:ddom-bound}
    If $M=(m_1,...,m_q)$, $d \leq q-1$, and $M$ is $d$-dominant, then $\pd(S/M) \geq d+1$
\end{proposition}

\begin{proof}
     By \cref{lem:scarf}, $\opb_j(S/M)=\binom{q}{j}$ for each $0\leq j \leq d$. Hence $\pd(S/M) \geq d$. If $\pd(S/M) = d$ then since $d\leq q-1$, $$\sum_{i\geq 0}(-1)^i\opb_i=\sum_{i=0}^d(-1)^i\opb_i = \sum_{i=0}^{d}(-1)^i\binom{q}{i}=(-1)^d\binom{q-1}{d} \neq 0$$ This is a contradiction by \cref{lem:alt-sum}, hence $\pd(S/M) > d$.
\end{proof}

For a monomial ideal $M=(m_1,...,m_q)$, we use the notation $\sigma$ to denote a subset of $\{m_1,...,m_q\}$ that defines a basis element from the Taylor resolution of $M$ that is also in the minimal resolution of $S/M$. We also define $a_\sigma = \lcm(\sigma)$. We will show that $d$-dominance also forces the quantity $\deg(a_\sigma)-|\sigma|$ to be monotonically non-decreasing through homological degree $d$. This means that in the case where $M$ is $(k-1)$-dominant, the regularity of $S/M$ is calculated by finding the maximal degree of the multidegrees of the size $k$-sets of generators that include the witness $m$. First, we make a formal definition.

\begin{definition}
    Consider a monomial ideal $M=(m_1,...,m_q)$. The regularity of $M$, denoted $\reg(S/M)$, is defined as follows: $\reg(S/M)=\max\{j-i:\opb_{i,j}(S/M) \neq 0\}$. For the ideals we have been considering, this is equivalent to: $\reg(S/M)=\max_\sigma\{\deg(a_\sigma)-|\sigma|\}$, where the maximum is taken over all possible basis elements $\sigma$ in the minimal free resolution of $M$.
\end{definition}

\begin{proposition}\label{prop:ddom-reg}
    If $M$ is $d$-dominant, then $\reg(S/M)=max\{\deg(a_\sigma)-|\sigma|:|\sigma| \geq d\}$ where $\sigma$ is the face supporting a basis element $[\sigma]$ in the minimal resolution and $a_\sigma=\lcm(\sigma)$. That is, the regularity lives in homological degree at least $d$ in the minimal resolution.
\end{proposition}

\begin{proof}
    Suppose $M$ is $d$-dominant.
    
    Claim: For each $0<j\leq d$ and $A \subset \{m_1,...,m_q\}$ such that $|A|=j$ and each $m_i \in (\{m_1,...,m_q\}\setminus A)$, we have $\deg(\lcm(A))<\deg(\lcm(A\cup\{m_i\}))$.

    Proof of claim: Let $0<j\leq d$ and consider an arbitrary set of generators $A \subset \{m_1,...,m_q\}$ of size $j$. Notice that for each $m_i\in (\{m_1,...,m_q\}\setminus A)$, we have that $m_i \nmid \lcm(A)$ by $d$-dominance. This means that $m_i$ must have at least one variable whose exponent is strictly greater than its exponent in $\lcm(A)$. Notice this means that $\deg(\lcm(A))< \deg(\lcm(A\cup \{m_i\}))$. This completes the proof of the claim.

    By \cref{lem:scarf} each face of size $\leq d$ is Scarf, hence the basis elements of the minimal resolution in homological degrees $\leq d$ are exactly the faces of those size. Let $\sigma$ be a basis element with $|\sigma|<d$. We can extend it one generator at a time, $\sigma \subset \sigma_1 \subset ... \subset \sigma_t$ where $|\sigma_t|=d$. At each stage $|\sigma_j|<d$, the claim entails that $\deg(\lcm(\sigma_{j+1})-|\sigma_{j+1}|) \geq \deg(\lcm(\sigma_j)-|\sigma_j|)$. Also, $\sigma_t$ is Scarf since it is size $d$, and thus is a basis element of the minimal resolution. This means that each basis element of size $<d$ is dominated by one with size $\geq d$, and hence $\reg(S/M)=max\{\deg(a_\sigma)-|\sigma|:|\sigma| \geq d\}$.
\end{proof}

When the level of dominance is exactly adjacent to the divisibility level, the inequalities shown above allow us to determine the projective dimension, all Betti numbers, and the differential maps. Moreover, the witness-first Lyubeznik resolution of monomial ideals of this form is their minimal resolution.

\begin{theorem}\label{thm:big-lyub}
    Suppose that $M=(m_1,...,m_q)$ is $k$-divisible and $(k-1)$-dominant. Then the witness-first Lyubeznik resolution $\mathbb{L}_M$ exists and is the minimal resolution of $S/M$. Moreover,\\
    (i) $\pd(S/M)=k$, \\
    (ii) for all $i \in \{0,1,...,k-1\}$, $\opb_i=\binom{q}{i}$ and $\opb_k=\binom{q-1}{k-1}$, \\
    (iii) $\reg(S/M)=\max\{\deg(a_\sigma)-k : |\sigma|=k, [\sigma]\in \mathbb{L}_M\}$.
\end{theorem}
\begin{proof}
    As usual, $M=(m,m_1,...m_{q-1})$ with $m$ a witness of $k$-divisibility. By a witness-first ordering of the generators we mean any ordering of the form: $m<m_{i_1}<...<m_{i_{q-1}}$, where $\{i_1,...,i_{q-1}\}=\{1,2,...,q-1\}$. We will fix an arbitrary witness-first ordering with $m$ as the least element, and observe that the rooted faces can be identified without reference to the ordering of $m_1,...,m_q$. Hence, in this setting we are justified in using the notation $\mathbb{L}_M$.
    
    Claim:  Fix a witness-first ordering: the rooted faces of $\Delta_M$ are exactly each face of size $\leq k-1$ and the size $k$ faces which include $m$. 
    
    Proof of claim:  Let $\sigma_1$ be a face of the full simplex of $M$ such that $|\sigma_1|\leq k-1$. Consider an arbitrary subface $\sigma_1'$ of $\sigma_1$, and notice that if $m|\lcm(\sigma_1')$ then $m \in \sigma_1'$ by $|\sigma_1'|$-dominance. Similarly if $m_i|\lcm(\sigma_1')$ then $m_i \in \sigma_1'$. Therefore, each subface of $\sigma_1$ is rooted, and hence so is $\sigma_1$. 
    
    Next, consider another face $\sigma_2$ with $|\sigma_2| =k$ and $m \in \sigma_2$. Consider a subface $\sigma_2'$ with $|\sigma_2'|<k$. Then as before $\sigma_2'$ must be rooted by $|\sigma_2'|$-dominance. Since the full face is also rooted because $m \in \sigma_2$, $\sigma_2$ is itself rooted. By \cref{lem:root}, each face of size strictly greater than $k$ is unrooted. Also, if $\sigma_3$ is s.t $|\sigma_3|=k$ and $m \notin \sigma_3$ then clearly $\sigma_3$ is not rooted since $m|\lcm(\sigma_3)$ and $m \notin \sigma_3$. Hence, each rooted face is as we claimed and the resolution is minimal since no two faces in adjacent homological degrees share a multidegree by \cref{lem:scarf}.
    
    Each face in homological degree $k$ in the minimal resolution includes $m$ and each face of size $k$ that includes $m$ which supports a basis element in the Taylor resolution is a rooted face. There are $\binom{q-1}{k-1}$ such faces, hence $\opb_k=\binom{q-1}{k-1}$. Each Taylor-face of size $<k$ is rooted, thus $\opb_j=\binom{q}{j}$ for each $0\leq j<k$.
   Since $k\leq q$, $\opb_k=\binom{q-1}{k-1}>0$ and also $\pd(S/M) \leq k$ by $k$-divisibility, hence $\pd(S/M)=k$.
   
  By \cref{prop:ddom-reg}, the regularity is computed over faces of size $k$ or $k-1$. We will show that each $(k-1)$-face is dominated by a rooted size-$k$ face. Let $\sigma$ be a face such that $|\sigma|=k-1$. If $m \notin \sigma$, consider $\tau = \sigma \cup \{m\}$. Now $(k-1)$-dominance forces $m \nmid \lcm(\sigma)$, and so $\deg(\tau)>\deg(\sigma)$. If $m \in \sigma$, consider any $m_i \notin \sigma$ and define $\tau= \sigma \cup \{m_i\}$. Then again, $m_i \nmid \lcm(\sigma)$ and hence $\deg(\tau)>\deg(\sigma)$. In either case, $\tau$ is a rooted face and hence $\deg(\lcm(\tau))-k \geq \deg(\lcm(\sigma))-(k-1)$. Hence, $\reg(S/M)=\max\{\deg(a_\sigma)-k : |\sigma|=k, [\sigma]\in \mathbb{L}_M\}$. 
\end{proof}

\begin{exam}
    We return to the ideal from the first example. Recall that $M=(xyz,xu,yv,zw)$ is $2$-dominant and $3$-divisible. Then by \cref{thm:big-lyub}, the Lyubeznik resolution obtained from the ordering $xyz<xu<yv<zw$ is the minimal resolution (any ordering which has $xyz$ as the least element produces a Lyubeznik resolution which is the minimal resolution. Thus, we could also use the ordering, $xyz<zw<xu<yv$.) The minimal resolution is thus well understood: producing the faces in each homological degree is a simple combinatorial exercise, the differential mappings are fully understood, etc. Also, $\pd(S/M)=3$, the Betti numbers are $\opb_0=\binom{4}{0}=1$, $\opb_1=\binom{4}{1}=4$, $\opb_2=\binom{4}{2}=6$, $\opb_3=\binom{4-1}{3-1}=3$, and  $\reg(S/M)=\max\{\deg(a_\sigma)-3 : |\sigma|=3, [\sigma]\in \mathbb{L}_M\}= 6-3=3$.
\end{exam}

\begin{remark}
   This theorem allows us to resolve Alesandroni's Conjecture $7.3$ from \cite{Ales1}, which says: if $M$ is minimally generated by more than one monomial, and there is a minimal generator that divides the lcm of every pair of generators, then $\pd(S/M ) = 2$. Notice that $M$ is $1$-dominant because every monomial ideal is, and also $M$ is not $2$-dominant since $M$ is $2$-divisible. The statement of the conjecture is equivalent to this: if $q \geq 2$ and $M$ is $2$-divisible, then $\pd(S/M)=2$. But $2$-divisibility implies $\pd(S/M)\leq 2$ by \cref{thm:kdiv-bound} and $1$-dominance implies $\pd(S/M)>1$ by \cref{prop:ddom-bound}, hence we know $\pd(S/M)=2$ as desired.
\end{remark}

\begin{definition}
  We must rehearse some definitions from \cite{Ales3}. By $\supp(m_i)$ we mean the set of variables that appear in $m_i$.  A set of monomials $\{m_1,...,m_q\}$ is called dominant if each generator $m_i$ is such that $\exists  x \in \supp(m_i)$ with $\deg_{m_i}(x) > \max\{\deg_{m}(x):m\in \{m_1,...,m_q\}\setminus\{m_i\} \}$, where $\deg_{m}(x)$ is the degree of the variable $x$ in the monomial $m$. That is, each generator must have a variable that appears with degree strictly greater than in any other generator in the set.    We call $\{m_1,...,m_q\}$ a semidominant set if it is not dominant but $\exists m_j \in \{m_1,...,m_q\}$ such that $\{m_1,...,\widehat{m_j},...,m_q\}$ is dominant. A monomial ideal $M=(m_1,...,m_q)$ is called dominant if $\{m_1,...,m_q\}$ is a dominant set and semidominant if $\{m_1,...,m_q\}$ is a semidominant set.
\end{definition}

Alesandroni defined the terminology of $1$-cancellation ideals, and we prove a characterization of these ideals in terms of our notion of $d$-dominance. Notice that, while our definition of $d$-dominance is different from a monomial ideal being dominant, several notions overlap: $M=(m_1,...,m_q)$ is a dominant ideal if and only if $M$ is $(q-1)$-dominant if and only if $M$ satisfies $\kappa(M)=q$. This is because $(q-1)$-dominance means that no generator divides the lcm of the rest, which can only happen if each generator has a variable whose degree is strictly higher than in any other generator of $M$. We will say more about how these ideas are related when we prove a characterization of $1$-cancellation ideals.

\begin{definition}
    A monomial ideal $M$ is called a $1$-cancellation ideal if a minimal resolution of $S/M$ can be obtained from $\mathbb{T}_M$ by means of exactly one consecutive cancellation.
\end{definition}

\begin{theorem} \label{thm:ales}
    Let $M$ be a monomial ideal minimally generated by $G = \{m_1,...,m_q\}$. Then $M$ is a $1$-cancellation ideal if and only if $G$ is not dominant but every subset $G\setminus \{m_i\}$ is.
\end{theorem}
\begin{proof}
    ($\implies$) Since $M$ is $1$-cancellation, it is not dominant and, hence, $G$ is not dominant. Also, if we assume a subset $G \setminus \{m_i\}$ is not dominant, there is a monomial $m_j \in G \setminus\{m_i\}$ such that $m_j| \lcm(G \setminus\{m_i,m_j\})$. It follows that $\lcm(G \setminus \{m_i,m_j\})= \lcm(G \setminus\{m_i\})$. Let $[\sigma]$ and $[\tau]$ be the Taylor symbols defined by $G \setminus\{m_i\}$ and $G\setminus\{m_i,m_j\}$, respectively. Then $[\sigma]$ and $[\tau]$ are face and facet of equal multidegree, and they appear in homological degrees $q-1$ and $q-2$. Thus, it is possible to make the cancellation $0 \xrightarrow{}S[\sigma] \xrightarrow{} S[\tau] \xrightarrow{} 0$ in $\mathbb{T}_M$, which contradicts the fact that the only cancellation occurs in the last two homological degrees of $\mathbb{T}_M$.
    
    ($\impliedby$) Suppose that $[\sigma]$ and $[\tau]$ are basis elements of $\mathbb{T}_M$, such that $[\tau]$ is a facet of $[\sigma]$ and they have equal multidegree. Then $\sigma$ is not a dominant set. If \#$(\sigma)<q$, then there is a set of the form $G \setminus \{m_i\}$ such that $\sigma \subset G \setminus\{m_i\}$. This implies that $G \setminus \{m_i\}$ is not dominant, which is a contradiction. Thus, \#$(\sigma)=q$.
\end{proof}

We prove that $M=(m_1,...,m_q)$ being a $1$-cancellation ideal turns out to be equivalent to the tight case of $(q-1)$-divisibility, $(q-2)$-dominance.

\begin{proposition}\label{prop:one-canc}
    $M=(m_1,...,m_q)$ is a $1$-cancellation ideal if and only if $M$ is $(q-1)$-divisible and $(q-2)$-dominant.
\end{proposition}
\begin{proof}
    ($\implies$) Suppose that  $M$ is a $1$-cancellation ideal. Recall \cref{thm:ales} states that $M$ is nondominant but $M_i=(m_1,...,\widehat{m_i},...,m_q)$ is dominant for each $i \in \{1,...,q\}$. This implies that $M$ is $(q-1)$-divisible since $M$ contains a nondominant element. Consider a $(q-2)$-collection $A=\{m_{i_1},...,m_{i_{q-2}}\}$ of generators and suppose for a contradiction that another generator $m_j\notin A$ is such that $m_j|\lcm(A)$. Notice that $A\cup \{m_j\}$ is a $q-1$ collection of minimal generators, hence it is dominant. This contradicts the assumption that $m_j|\lcm(A)$, and hence $M$ must be $(q-2)$-dominant.
    
    ($\impliedby$) If $M$ is $(q-1)$-divisible and $(q-2)$-dominant, then by \cref{thm:big-lyub}, $\pd(S/M)=q-1$. Each face strictly below homological degree $q-1$ is Scarf by \cref{lem:scarf}, hence the consecutive cancellations necessary to obtain the minimal resolution from the Taylor resolution must occur between homological degrees $q-1$ and $q$. You may only perform one cancellation since doing so would remove the only Taylor-face of homological degree $q$, and then no further cancellations are possible by the Scarfness of faces strictly below homological degree $q-1$ from \cref{lem:scarf}.
    
\end{proof}

Next we show that for $k$-divisible monomial ideals, it is enough to check the last Betti number to determine whether $M$ is $(k-1)$-dominant.

\begin{proposition} \label{prop:betti-char}
    If $M$ is $k$-divisible, then $M$ is $(k-1)$-dominant if and only if $\opb_k=\binom{q-1}{k-1}$.
\end{proposition}

\begin{proof}
    ($\implies$) By \cref{thm:big-lyub}, if $M$ is $k$-divisible and $(k-1)$-dominant then $\opb_k(S/M)=\binom{q-1}{k-1}$.
    
    ($\impliedby$)  We prove the contrapositive. Suppose $M$ is $k$-divisible and not $(k-1)$-dominant. We know from \cref{thm:kdiv-bound} that $k$-divisibility implies $\opb_k\leq \binom{q-1}{k-1}$, so we must simply show that equality is impossible. Since $M$ is not $(k-1)$-dominant, there must be some minimal generator $w$ that divides the lcm of some other $k-1$ collection of minimal generators: $\sigma_{k-1}=\{m_{i_1},...,m_{i_{k-1}}\}$ (notice: different means $w \notin \{m_{i_1},...,m_{i_{k-1}}\}$). By $k$-divisibility there is a witness $m$ such that $m|\lcm(\sigma_{k-1},w)$. If $m\neq w$, fix any order in which $m$ is least and $w$ is second (the rest of the order does not matter). Else, just choose the generic witness first ordering from \cref{thm:big-lyub}. We now split into cases. 
    
    \textbf{Case 1: $m \notin (\sigma_{k-1}\cup \{w\})$}
    
    The face $\sigma_{k-1}$ is not rooted in the Lyubeznik complex formed by the ordering  $m<w<m_{i_1}<...<m_{i_{q-2}}$ since $w|\lcm(\sigma_{k-1})$ and $w \notin \sigma_{k-1}$. This means that the face $\sigma_{k-1} \cup \{m\}$ is also not rooted, since $\sigma_{k-1}$ as a subface is not rooted. This means $\opb_k$ is strictly smaller than $\binom{q-1}{k-1}$. 
    
    \textbf{Case $2$: $m \in (\sigma_{k-1}\cup \{w\})$}
    
    By \cref{lem:root}, each size-$k$ face not containing $m$ is not rooted. If $m=w$, $m|\lcm(\sigma_{k-1})$ but $m \notin \sigma_{k-1}$ and hence $\sigma_{k-1}$ is not rooted and so neither is $\sigma_{k-1}\cup \{m\}$. If $m \in \sigma_{k-1}$, then $\lcm(\sigma_{k-1}\cup \{w\})=\lcm(\sigma_{k-1})$. If $\sigma_{k-1}\cup \{w\}$ is unrooted, done. If $\sigma_{k-1}\cup \{w\}$ is rooted then so also is $\sigma_{k-1}$. But then $\sigma_{k-1}$ is a facet of $\sigma_{k-1}\cup \{w\}$ and hence we may perform a consecutive cancellation, reducing $\opb_k<\binom{q-1}{k-1}$.
\end{proof}

\begin{exam}
     This serves as a nice example of a $k$-divisible, $(k-1)$-dominant type ideal, in particular of one that is $3$-divisible and $2$-dominant with $q=6$ minimal generators. Let $z=\{(i,j)\in\{1,2,...,5\}^2 \ | \ i \neq j\}$ and $y_p = \{(i,j)\in z \ | \ i \neq p, \ j \neq p\} \cup \{(p,p)\}$. Then define $$\overline{z}=\prod_{e\in z}x_e$$ And $$\overline{y_p}=\prod_{e\in y_p}x_e$$ And $M = (\overline{z},\overline{y_1},\overline{y_2},\overline{y_3},\overline{y_4},\overline{y_5})$.

     Then $\overline{z}|\lcm(\overline{y_i},\overline{y_j},\overline{y_k})$ for each choice of a $3$-set: for any $\{r,s,t\}\subset \{1,2,3,4,5\}$, each variable $x_{(a,b)}$ is omitted only by $\overline{y_a}$ and $\overline{y_b}$. Hence, since we have chosen a $3$-set every variable in $\overline{z}$ must be in $\lcm(\overline{y_r},\overline{y_s},\overline{y_t})$. \\
     Now for $2$-dominance we must check both types of generator. First, $\overline{z}$ fails to divide any pair $\overline{y_r},\overline{y_s}$ since $x_{(r,s)}\notin \lcm (\overline{y_r},\overline{y_s})$ but $x_{(r,s)} \in \overline{z}$. Also, $\overline{y_r} \nmid \lcm(\overline{y_s},\overline{y_t})$ since again $x_{(s,t)}\notin \lcm (\overline{y_s},\overline{y_t})$ but $x_{(s,t)} \in \overline{y_r}$. Moreover, $\overline{y_r} \nmid \lcm(\overline{z},\overline{y_s})$ since $x_{(r,r)}\in \overline{y_r}$ and $x_{(r,r)}\notin \lcm(\overline{z},\overline{y_s})$.

     Since $k= 3$ and $d=2$, by \cref{thm:big-lyub} we also know that $\pd(S/M)=3$, $\opb_0(S/M)=1,\opb_1(S/M)=6,\opb_2(S/M)=15,\opb_3(S/M)=10$, and $\reg(S/M)=19$. It is also possible to extract the differential maps since the minimal resolution is explicitly given in \cref{thm:big-lyub}, but this is very tedious and so is not included here.
\end{exam}

\section{Construction of $(q,k,d)$-ideals}

How do we know that the results proven in the previous section describe an infinite class of monomial ideals? For which triples $(q,k,d)$ with $q,k,d \in \mathbb{N}$ and $1\leq d <k\leq q-1$ does there exist a monomial ideal $M$ such that $M$ has $q$ minimal generators, $\kappa(M) = k$, and $\delta(M)=d$? We prove that for any choice of $q,k,d$ in that range there exists a monomial ideal with those properties.

\begin{remark}
        For arbitrary $q \geq 3$ and $1 \leq d <k\leq q-1$, we will provide a construction method for a square-free, semidominant monomial ideal with the properties that: $M$ has $q$ minimal generators, $\delta(M)=d$ and $\kappa(M)=k$. If we view each generator as the set of its variables, notice that $d$-dominance means that the set of minimal generators must be $d$-cover-free. By $d$-coverfree, we mean that no generator should be a subset of any $d$-collection of other generators. For each generator that has a unique variable (the ``dominant" generators), this property is automatic. For the nondominant generator $m$, each $d$-collection of other generators must be without (must "miss") at least one variable in $m$. To guarantee $k$-divisibility we must have each variable in $m$ misses at most $k-1$ dominant generators. If $k$ is to be minimal, there must be a variable missing exactly $k-1$ other generators. Of course, the dominant generators must also each possess a private variable (by a ``private variable", we mean a variable that is used by exactly one generator). These ideas are illustrated by a construction.
\end{remark}

\begin{construct}
    We construct $M=(m,m_1,...,m_{q-1})$, with $M$ square-free and semidominant having $m$ as the nondominant element, which, for a fixed $k,d$ satisfying $1\leq d<k\leq q-1$, has $\delta(M)=d$ and $\kappa(M)=k$. Consider the $d$-subsets of $\{1,2,...,q-1\}$, which number $\binom{q-1}{d}$. For each of these subsets $A$, create a variable $x_A$. We require one more variable to ensure that $\kappa(M)=k$, that is, to ensure that $M$ is not $k-1$ divisible. This variable, say $z$, is given to $m$ and to each other generator except exactly $k-1$ of them. Call this set of $k-1$ generators $B$ (it can be chosen arbitrarily, as long as it is of that size). Then define $$m=z \cdot\prod_A x_A$$ And, for each $i \in B$, define: $$m_i=p_i \cdot \prod_{A:i \notin A}x_A$$ And for $j \notin B$: $$m_j=z\cdot p_j \cdot \prod_{A:j \notin A}x_A$$ This requires exactly $(q)+\binom{q-1}{d}$ variables. Each $d$-subset of $\{m_1,...,m_{q-1}\}$ misses some variable in $m$, hence $M$ satisfies $\delta(M)=d$. It is not possible that $M$ is $(d+1)$-dominant because you may choose one monomial with $z$, and then pick arbitrarily $d$ more (notice: $d+1\leq q-1$). Each variable in $m$ will appear, and hence $m$ will divide the lcm. Also $\kappa(M)=k$ because each variable in $m$ must appear in any $k$-collection of other minimal generators, and $M$ not $(k-1)$-divisible because you can choose the $k-1$ monomials which do not contain $z$, and then $m$ will not divide their lcm. And no $m_i$ can be a witness: since $k-1 \leq q-2$ there must be a $(k-1)$-subset of the generators that avoids $m_i$ with lcm not divisible by the private variable $p_i$, hence also not divisible by $m_i$. Lastly, obviously no generator may divide another, hence $q$ is the size of its minimal generating set. This completes the construction.
\end{construct}

\begin{exam}
    We will illustrate this method with an example. Consider: $(q,k,d)=(5,3,2)$, so $M=(m,m_1,m_2,m_3,m_4)$. First, each dominant monomial $m_1,...,m_4$ is given a private variable. Give these variables names: $p_i \in \supp(m_i)$. Now for each of the six $2$-subsets of $\{1,2,3,4\}$, we will create a variable. Consider $x_{\{1,2\}},x_{\{1,3\}},x_{\{1,4\}},x_{\{2,3\}},x_{\{2,4\}},x_{\{3,4\}}$. Then each variable $x_{\{i,j\}}$ is given to each $m_s$ with $i \neq s \neq j$. Finally we define the variable $z$, which we will use in $m$ and $2$ other generators, say $m_1$ and $m_2$. For ease of reading, we relabel the variables with set-subscripts. Define $(x_{\{1,2\}},x_{\{1,3\}},x_{\{1,4\}},x_{\{2,3\}},x_{\{2,4\}},x_{\{3,4\}})=(y_1,y_2,y_3,y_4,y_5,y_6)$, Our construction is then: $$M= (m,m_1,m_2,m_3,m_4) =(y_1y_2y_3y_4y_5y_6z,p_1y_4y_5y_6z,p_2y_2y_3y_6z,p_3y_1y_3y_5,p_4y_1y_2y_4)$$
    The monomial $y_1y_2y_3y_4y_5y_6z$ divides the lcm of any $3$ monomials from $$\{y_1y_2y_3y_4y_5y_6z,p_1y_4y_5y_6z,p_2y_2y_3y_6z,p_3y_1y_3y_5,p_4y_1y_2y_4\}$$ and no monomial divides the lcm of any $2$ others.
\end{exam}

\textbf{Acknowledgments:} We wish to recognize and thank the people who have been involved in various ways with this paper. In particular, thanks is due to Christopher Chin for the elegant construction used in example $3.15$, and Noah Ripke, Abigail Ho, and Phillip Yoon for helpful discussion. Our work was largely done at the California State University at Chico, so thanks is also due to the mathematics and statistics department there, especially to Drs. Kevin McGown and John Lind. This work was funded by NSF grant: $DMS-2244020$.


\begin{bibdiv}
\begin{biblist}

\bib{Ales1}{article}{
  author={Alesandroni, Guillermo},
  title={Structural Decomposition of Monomial Resolutions},
  journal={J. Algebra},
  volume={516},
  year={2018},
  pages={210--244},
  review={\MR{3864455}},
  doi={10.1016/j.jalgebra.2018.09.003}
}


\bib{Ales3}{article}{
  author={Alesandroni, Guillermo},
  title={Minimal resolutions of dominant and semidominant ideals},
  journal={J. Pure Appl. Algebra},
  volume={221},
  number={4},
  year={2017},
  pages={780--798},
  review={\MR{3585040}},
  doi={10.1016/j.jpaa.2016.08.002}
}



\bib{Merm}{collection.article}{
  author={Mermin, Jeff},
  booktitle={Progress in Commutative Algebra 1},
  title={Three Simplicial Resolutions},
  pages={127--141},
  publisher={De Gruyter},
  address={Berlin},
  year={2012},
  review={\MR{2931745}},
}


\bib{Peeva}{book}{
  author={Peeva, I.},
  title={Graded Syzygies},
  series={Algebra and Applications},
  volume={14},
  publisher={Springer},
  address={London},
  year={2010}
}

\end{biblist}
\end{bibdiv}
\end{document}